\documentclass[11pt,reqno]{amsart}

\usepackage{caption,amsmath,amssymb,amsfonts,amsthm,latexsym,graphicx,multirow,color,hyperref,enumerate}
\usepackage[all]{xy}

\numberwithin{table}{section}

\newcommand{\A}{\mathrm{A}}    \newcommand{\Aut}{\mathrm{Aut}}
 \newcommand{\bbF}{\mathbb{F}} 
         
\newcommand{\D}{\mathrm{D}} 

\newcommand{\F}{\mathrm{F}} 
\newcommand{\G}{\mathrm{G}}       \newcommand{\GL}{\mathrm{GL}} \newcommand{\GO}{\mathrm{O}}

 \newcommand{\magma}{\textsc{Magma}}
\newcommand{\N}{\mathrm{N}}

\newcommand{\Pa}{\mathrm{P}}         \newcommand{\PSL}{\mathrm{PSL}}    \newcommand{\PSp}{\mathrm{PSp}}

  \newcommand{\SL}{\mathrm{SL}}  \newcommand{\Soc}{\mathrm{Soc}} \newcommand{\Sp}{\mathrm{Sp}}  \newcommand{\SU}{\mathrm{SU}}  \newcommand{\Sy}{\mathrm{S}}

\def\Z{{\bf Z}}

\newtheorem{theorem}{Theorem}[section]
\newtheorem{lemma}[theorem]{Lemma}

\newtheorem{problem}[theorem]{Problem}

\theoremstyle{definition}

\newtheorem*{remark}{Remark}

\begin{document}

\title[Factorizations of almost simple orthogonal groups in odd dimension]{Factorizations of almost simple orthogonal groups in odd dimension}

\author[Li]{Cai Heng Li}
\address{(Li) SUSTech International Center for Mathematics and Department of Mathematics\\Southern University of Science and Technology\\Shenzhen 518055\\Guangdong\\P.~R.~China}
\email{lich@sustech.edu.cn}

\author[Wang]{Lei Wang}
\address{(Wang) School of Mathematics and Statistics\\Yunnan University\\Kunming 650091\\Yunnan\\P.~R.~China}
\email{wanglei@ynu.edu.cn}

\author[Xia]{Binzhou Xia}
\address{(Xia) School of Mathematics and Statistics\\The University of Melbourne\\Parkville 3010\\VIC\\Australia}
\email{binzhoux@unimelb.edu.au}

\begin{abstract}
This is the third one in a series of papers classifying the factorizations of almost simple groups with nonsolvable factors.
In this paper we deal with orthogonal groups in odd dimension.


\end{abstract}

\maketitle

\section{Introduction}

An expression $G=HK$ of a group $G$ as the product of subgroups $H$ and $K$ is called a \emph{factorization} of $G$, where $H$ and $K$ are called \emph{factors}. A group $G$ is said to be \emph{almost simple} if $S\leqslant G\leqslant\Aut(S)$ for some nonabelian simple group $S$, where $S=\Soc(G)$ is the \emph{socle} of $G$. In this paper, by a factorization of an almost simple group we mean that none its factors contains the socle. The main aim of this paper is to solve the long-standing open problem:

\begin{problem}\label{PrbXia1}
Classify factorizations of finite almost simple groups.
\end{problem}

Determining all factorizations of almost simple groups is a fundamental problem in the theory of simple groups, which was proposed by Wielandt~\cite[6(e)]{Wielandt1979} in 1979 at The Santa Cruz Conference on Finite Groups. It also has numerous applications to other branches of mathematics such as combinatorics and number theory, and has attracted considerable attention in the literature.
In what follows, all groups are assumed to be finite if there is no special instruction.

The factorizations of almost simple groups of exceptional Lie type were classified by Hering, Liebeck and Saxl~\cite{HLS1987}\footnote{In part~(b) of Theorem~2 in~\cite{HLS1987}, $A_0$ can also be $\G_2(2)$, $\SU_3(3)\times2$, $\SL_3(4).2$ or $\SL_3(4).2^2$ besides $\G_2(2)\times2$.} in 1987.
For the other families of almost simple groups, a landmark result was achieved by Liebeck, Praeger and Saxl~\cite{LPS1990} thirty years ago, which classifies the maximal factorizations of almost simple groups. (A factorization is said to be \emph{maximal} if both the factors are maximal subgroups.)
Then factorizations of alternating and symmetric groups are classified in~\cite{LPS1990}, and factorizations of sporadic almost simple groups are classified in~\cite{Giudici2006}.
This reduces Problem~\ref{PrbXia1} to the problem on classical groups of Lie type.
Recently, factorizations of almost simple groups with a factor having at least two nonsolvable composition factors are classified in~\cite{LX2019}\footnote{In Table~1 of~\cite{LX2019}, the triple $(L,H\cap L,K\cap L)=(\Sp_{6}(4),(\Sp_2(4)\times\Sp_{2}(16)).2,\G_2(4))$ is missing, and for the first two rows $R.2$ should be $R.P$ with $P\leqslant2$.}, and those with a factor being solvable are described in~\cite{LX} and~\cite{BL2021}.

As usual, for a finite group $G$, we denote by $G^{(\infty)}$ the smallest normal subgroup of $X$ such that $G/G^{(\infty)}$ is solvable.
For factorizations $G=HK$ with nonsolvable factors $H$ and $K$ such that $L=\Soc(G)$ is a classical group of Lie type, the triple $(L,H^{(\infty)},K^{(\infty)})$ is described in~\cite{LWX}. Based on this work, in the present paper we characterize the triples $(G,H,K)$ such that $G=HK$ with $H$ and $K$ nonsolvable, and $G$ is an orthogonal group in odd dimension.

For groups $H,K,X,Y$, we say that $(H,K)$ contains $(X,Y)$ if $H\geqslant X$ and $K\geqslant Y$, and that $(H,K)$ \emph{tightly contains} $(X,Y)$ if in addition $H^{(\infty)}=X^{(\infty)}$ and $K^{(\infty)}=Y^{(\infty)}$. 
Our main result is the following Theorem~\ref{ThmOmega}.
Note that it is elementary to determine the factorizations of $G/L$ as this group has relatively simple structure (and in particular is solvable).

\begin{theorem}\label{ThmOmega}
Let $G$ be an almost simple group with socle $L=\Omega_{2m+1}(q)$, where $q$ is odd and $m\geqslant3$, and let $H$ and $K$ be nonsolvable subgroups of $G$ not containing $L$. Then $G=HK$ if and only if (with $H$ and $K$ possibly interchanged) $G/L=(HL/L)(KL/L)$ and $(H,K)$ tightly contains $(X^\alpha,Y^\alpha)$ for some $(X,Y)$ in Table~$\ref{TabOmega}$ and $\alpha\in\Aut(L)$.
\end{theorem}

\begin{remark}
Here are some remarks on Tables~\ref{TabOmega}:
\begin{enumerate}[{\rm(I)}]
\item The column $Z$ gives the smallest almost simple group with socle $L$ that contains $X$ and $Y$. In other words, $Z=\langle L,X,Y\rangle$.
It turns out that $Z=XY$ for all pairs $(X,Y)$.
\item The groups $X$, $Y$ and $Z$ are described in the corresponding lemmas whose labels are displayed in the last column. 
\item The description of groups $X$ and $Y$ are up to conjugations in $Z$ (see Lemma~\ref{LemXia04}(b) and Lemma~\ref{LemXia03}).
\end{enumerate}
\end{remark}

\begin{table}[htbp]
\captionsetup{justification=centering}
\caption{$(X,Y)$ for orthogonal groups in odd dimension}\label{TabOmega}
\begin{tabular}{|l|l|l|l|l|l|}
\hline
Row & $Z$ & $X$ & $Y$ & Lemma\\
\hline
1 & $\Omega_{2m+1}(q)$ & $(q^{m(m-1)/2}.q^m){:}\SL_a(q^b)$ ($m=ab$), & $\Omega_{2m}^-(q)$ & \ref{LemOmega02}\\
 & & $(q^{m(m-1)/2}.q^m){:}\Sp_a(q^b)$ ($m=ab$) & & \\
 \hline
2 & $\Omega_7(q)$ & $\Omega_6^+(q)$, $\Omega_6^-(q)$, $\Omega_5(q)$, & $\G_2(q)$ & \ref{LemOmega03}, \ref{LemOmega04}, \ref{LemOmega11},\\
 & & $q^5{:}\Omega_5(q)$, $q^4{:}\Omega_4^-(q)$ & & \ref{LemOmega05}, \ref{LemOmega06}\\ 
\hline
3 & $\Omega_7(3^f)$ & $\SU_3(3^f)$, ${^2}\G_2(3^f)'$ ($f$ odd) & $\Omega_6^+(3^f)$ & \ref{LemOmega07}, \ref{LemOmega08}\\
4 & $\Omega_7(3^f)$ & $\SL_3(3^f)$ & $\Omega_6^-(3^f)$ & \ref{LemOmega07}\\
\hline
5 & $\Omega_{13}(3^f)$ & $\PSp_6(3^f)$ & $\Omega_{12}^-(3^f)$ & \ref{LemOmega09}\\
6 & $\Omega_{25}(3^f)$ & $\F_4(3^f)$ & $\Omega_{24}^-(3^f)$ & \ref{LemOmega10}\\
\hline
7 & $\Omega_7(3)$ & $3^4{:}\Sy_5$, $3^5{:}2^4{:}\A_5$, $3^4{:}\A_6$ & $\G_2(3)$ & \ref{LemOmega12}, \ref{LemOmega13}\\
8 & $\Omega_7(3)$ & $3^3{:}\SL_3(3)$, $\Omega_6^+(3)$, $\G_2(3)$ & $\A_9$, $\Sp_6(2)$ & \ref{LemOmega14}, \ref{LemOmega15}, \ref{LemOmega16}\\
9 & $\Omega_7(3)$ & $2^6{:}\A_7$, $\A_8$, $\A_9$, $2.\PSL_3(4)$, $\Sp_6(2)$ & $3^{3+3}{:}\SL_3(3)$ & \ref{LemOmega17}\\
10 & $\Omega_9(3)$ & $3^{6+4}{:}2.\Sy_5$, $3^{6+4}{:}8.\A_5$, $3^{6+4}{:}2^{1+4}.\A_5$ & $\Omega_8^-(3)$ & \ref{LemOmega18}\\
11 & $\Omega_{13}(3)$ & $3^{15+6}{:}\SL_2(13)$ & $\Omega_{12}^-(3)$ & \ref{LemOmega19}\\
\hline
\end{tabular}
\vspace{3mm}
\end{table}

\section{Preliminaries}

In this section we collect some elementary facts regarding group factorizations.

\begin{lemma}\label{LemXia01}
Let $G$ be a group, let $H$ and $K$ be subgroups of $G$, and let $N$ be a normal subgroup of $G$. Then $G=HK$ if and only if $HK\supseteq N$ and $G/N=(HN/N)(KN/N)$.
\end{lemma}

\begin{proof}
If $G=HK$, then $HK\supseteq N$, and taking the quotient modulo $N$ we obtain
\[
G/N=(HN/N)(KN/N).
\]
Conversely, suppose that $HK\supseteq N$ and $G/N=(HN/N)(KN/N)$. Then 
\[
G=(HN)(KN)=HNK 
\]
as $N$ is normal in $G$. Since $N\subseteq HK$, it follows that $G=HNK\subseteq H(HK)K=HK$, which implies $G=HK$.
\end{proof}

Let $L$ be a nonabelian simple group. We say that $(H,K)$ is a \emph{factor pair} of $L$ if $H$ and $K$ are subgroups of $\Aut(L)$ such that $HK\supseteq L$. For an almost simple group $G$ with socle $L$ and subgroups $H$ and $K$ of $G$, Lemma~\ref{LemXia01} shows that $G=HK$ if and only if $G/L=(HL/L)(KL/L)$ and $(H,K)$ is a factor pair.
As the group $G/L$ has a simple structure (and in particular is solvable), it is elementary to determine the factorizations of $G/L$.
Thus to know all the factorizations of $G$ is to know all the factor pairs of $L$.
Note that, if $(H,K)$ is a factor pair of $L$, then any pair of subgroups of $\Aut(L)$ containing $(H,K)$ is also a factor pair of $L$.
Hence we have the following:

\begin{lemma}\label{LemXia02}
Let $G$ be an almost simple group with socle $L$, and let $H$ and $K$ be subgroups of $G$ such that $(H,K)$ contains some factor pair of $L$. Then $G=HK$ if and only if $G/L=(HL/L)(KL/L)$.
\end{lemma}

In light of Lemma~\ref{LemXia02}, the key to determine the factorizations of $G$ with nonsolvable factors is to determine the minimal ones (with respect to the containment) among factor pairs of $L$ with nonsolvable subgroups.

\begin{lemma}\label{LemXia03}
Let $L$ be a nonabelian simple group, and let $(H,K)$ be a factor pair of $L$.
Then $(H^\alpha,K^\alpha)$ and $(H^x,K^y)$ are factor pairs of $L$ for all $\alpha\in\Aut(L)$ and $x,y\in L$.
\end{lemma}

\begin{proof}
It is evident that $H^\alpha K^\alpha=(HK)^\alpha\supseteq L^\alpha=L$. Hence $(H^\alpha,K^\alpha)$ is a factor pair.
Since $xy^{-1}\in L\subseteq HK$, there exist $h\in H$ and $k\in K$ such that $xy^{-1}=hk$. Therefore, 
\[
H^xK^y=x^{-1}Hxy^{-1}Ky=x^{-1}HhkKy=x^{-1}HKy\supseteq x^{-1}Ly=L,
\]
which means that $(H^x,K^y)$ is a factor pair.
\end{proof}

The next lemma is~\cite[Lemma~2(i)]{LPS1996}.

\begin{lemma}\label{LemXia05}
Let $G$ be an almost simple group with socle $L$, and let $H$ and $K$ be subgroups of $G$ not containing $L$. If $G=HK$, then $HL\cap KL=(H\cap KL)(K\cap HL)$. 
\end{lemma}

The following lemma implies that we may consider specific representatives of a conjugacy class of subgroups when studying factorizations of a group.

\begin{lemma}\label{LemXia04}
Let $G=HK$ be a factorization. Then for all $x,y\in G$ we have $G=H^xK^y$ with $H^x\cap K^y\cong H\cap K$.
\end{lemma}
  
\begin{proof}
As $xy^{-1}\in G=HK$, there exists $h\in H$ and $k\in K$ such that $xy^{-1}=hk$. Thus
\[
H^xK^y=x^{-1}Hxy^{-1}Ky=x^{-1}HhkKy=x^{-1}HKy=x^{-1}Gy=G,
\]
and
\[
H^x\cap K^y=(H^{xy^{-1}}\cap K)^y\cong H^{xy^{-1}}\cap K=H^{hk}\cap K=H^k\cap K=(H\cap K)^k\cong H\cap K.\qedhere
\]
\end{proof}

\section{Notation}

Throughout this paper, let $q=p^f$ be a power of an odd prime $p$, let $m\geqslant3$ be an integer, let $V$ be a vector space of dimension $2m+1$ over $\bbF_q$ equipped with a nondegenerate symmetric bilinear form $\beta$, let $e_1,f_1,\dots,e_m,f_m,d$ be a standard basis for $V$ as in~\cite[2.2.3]{LPS1990}, let $U=\langle e_1,\dots,e_m\rangle_{\bbF_q}$, let $U_1=\langle e_2,\dots,e_m\rangle_{\bbF_q}$, and let $W=\langle f_1,\dots,f_m\rangle_{\bbF_q}$. 
From~\cite[3.7.4]{Wilson2009} we see that $\Omega(V)_U=\Pa_m[\Omega(V)]$ has a subgroup $R{:}T$, where 
\[
R=q^{m(m-1)/2}.q^m 
\]
is the kernel of $\Omega(V)_U$ acting on $U$, and 
\[
T=\SL_m(q)
\] 
stabilizes both $U$ and $W$ (the action of $T$ on $U$ determines that on $W$ in the way described in~\cite[Lemma~2.2.17]{BG2016}).

Denote by $\bbF_q^\square$ the subgroup of index $2$ in $\bbF_q^\times$, denote $\bbF_q^\boxtimes=\bbF_q^\times\setminus\bbF_q^\square$, let $\lambda\in\bbF_q^\boxtimes$, let $v=e_1+\lambda f_1$, and let $\perp$ denote the perpendicularity with respect to $\beta$. 
Then $v^\perp$ is a $(2m)$-dimensional orthogonal space of minus type.
For a vector $u\in V$, denote by $r_u$ the reflection in $u$.
We follow the convention to denote $\SL_n^+(q)=\SL_n(q)$ and $\SL_n^-(q)=\SU_n(q)$.

\section{Infinite families of $(X,Y)$ in Table~\ref{TabOmega}}\label{SecOmega01}

In this section we construct the infinite families of factor pairs $(X,Y)$ in Table~\ref{TabOmega}.

\begin{lemma}\label{LemOmega01}
Let $G=\Omega(V)=\Omega_{2m+1}(q)$, let $M=R{:}T$, and let $K=G_v$. Then the following statements hold:
\begin{enumerate}[{\rm (a)}]
\item the induced group by the action of $M\cap K$ on $U$ is $\SL(U)_{U_1,e_1+U_1}=q^{m-1}{:}\SL_{m-1}(q)$;
\item the kernel of $M\cap K$ acting on $U$ is the special $p$-group $R\cap K=q^{(m-1)(m-2)/2}.q^{m-1}$;
\item if $H=R{:}S$ with $S\leqslant T$, then $H\cap K=(R\cap K).S_{U_1,e_1+U_1}$.
\end{enumerate}
\end{lemma}

\begin{proof} 
We first calculate $R\cap K$, the kernel of $M\cap K$ acting on $U$. For each $r\in R\cap K$, since $r$ fixes $e_1$ and $v$, we deduce that $r$ fixes $\langle e_1,v\rangle_{\bbF_{q^2}}=\langle e_1,f_1\rangle_{\bbF_q}$ pointwise. 
Hence $R\cap K$ is isomorphic to the pointwise stabilizer of $U_1$ in $\Omega(\langle e_2,f_2,\dots,e_m,f_m,d\rangle_{\bbF_q})$.
Then it is shown in~\cite[3.7.4]{Wilson2009} that $R\cap K=q^{(m-1)(m-2)/2}.q^{m-1}$ is special.

Since $K$ fixes $v$, it stabilizes $v^\perp$. 
Hence $M\cap K$ stabilizes $U\cap v^\perp=\langle e_2,\dots,e_m\rangle_{\bbF_q}=U_1$.
For arbitrary $h\in M\cap K$, write $e_1^h=\mu e_1+e$ with $e\in U_1$. Then
\[
(\lambda f_1)^h=(v-e_1)^h=v^h-e_1^h=v-(\mu e_1+e)=(1-\mu)e_1-e+\lambda f_1,
\] 
and so $\beta(e_1^h,(\lambda f_1)^h)=\beta(\mu e_1,\lambda f_1)$. Since $\lambda\neq0$ and $h$ preserves $\beta$, we obtain $\mu=1$.
It follows that $M\cap K$ stabilizes $e_1+U_1$, and so the induced group of $M\cap K$ on $U$ is contained in $\SL(U)_{U_1,e_1+U_1}$, that is, $(M\cap K)^U\leqslant\SL(U)_{U_1,e_1+U_1}=q^{m-1}{:}\SL_{m-1}(q)$. Now 
\[
M\cap K=(R\cap K).(M\cap K)^U\leqslant(R\cap K).\SL(U)_{U_1,e_1+U_1},
\] 
while
\begin{align*}
|M\cap K|\geqslant\frac{|M||K|}{|G|}&=\frac{|(q^{m(m-1)/2}.q^m){:}\SL_m(q)||\Omega_{2m}^-(q)|}{|\Omega_{2m+1}(q)|}\\
&=q^{(m+2)(m-1)/2}|\SL_{m-1}(q)|=|R\cap K||\SL(U)_{U_1,e_1+U_1}|.
\end{align*}
Thus we obtain $(M\cap K)^U=\SL(U)_{U_1,e_1+U_1}=q^{m-1}{:}\SL_{m-1}(q)$.

Finally, let $H=R{:}S$ with $S\leqslant T$. Since $(M\cap K)^U=\SL(U)_{U_1,e_1+U_1}=(M_{U_1,e_1+U_1})^U$, we obtain $M_{U_1,e_1+U_1}=(M\cap K)R$ as $R$ is the kernel of $M$ acting on $U$. This implies that $H_{U_1,e_1+U_1}=(H\cap K)R$, and so
\[
(H\cap K)/(R\cap K)\cong(H\cap K)R/R=H_{U_1,e_1+U_1}/R=S_{U_1,e_1+U_1}R/R\cong S_{U_1,e_1+U_1}.
\]
Therefore, $H\cap K=(R\cap K).S_{U_1,e_1+U_1}$.
\end{proof}

The following lemma gives the factor pairs $(X,Y)$ in Row~1 of Table~\ref{TabOmega}.

\begin{lemma}\label{LemOmega02}
Let $Z=\Omega(V)=\Omega_{2m+1}(q)$, let $X=R{:}S$ with $S\leqslant T$ such that $S=\SL_a(q^b)$ or $\Sp_a(q^b)$, where $m=ab$, and let $Y=Z_v$. Then 
\[
X\cap Y=
\begin{cases}
(q^{(m-1)(m-2)/2}.q^{m-1}.q^{m-b}){:}\SL_{a-1}(q^b)&\textup{if }S=\SL_a(q^b)\\
(q^{(m-1)(m-2)/2}.q^{m-1}.[q^{m-b}]){:}\Sp_{a-2}(q^b)&\textup{if }S=\Sp_a(q^b),
\end{cases}
\]
and $Z=XY$ with $X=(q^{m(m-1)/2}.q^m){:}S$ and $Y=\Omega_{2m}^-(q)$.
\end{lemma}

\begin{proof} 
It is clear that $X=(q^{m(m-1)/2}.q^m){:}S$ and $Y=\Omega_{2m}^-(q)$.
From Lemma~\ref{LemOmega01} we obtain $X\cap Y=(q^{(m-1)(m-2)/2}.q^{m-1}).S_{U_1,e_1+U_1}$. 
Moreover, by~\cite[Lemmas~4.1,~4.2 and~4.4]{LWX-Linear} we have
\[
S_{U_1,e_1+U_1}=
\begin{cases}
q^{m-b}{:}\SL_{a-1}(q^b)&\textup{if }S=\SL_a(q^b)\\
[q^{m-b}]{:}\Sp_{a-2}(q^b)&\textup{if }S=\Sp_a(q^b).
\end{cases}
\]
This proves the conclusion of this lemma on $X\cap Y$, and implies that
\[
\frac{|X|}{|X\cap Y|}=\frac{|(q^{m(m-1)/2}.q^m){:}S|}{|(q^{(m-1)(m-2)/2}.q^{m-1}).S_{U_1,e_1+U_1}|}=q^m(q^m-1)=\frac{|\Omega_{2m+1}(q)|}{|\Omega_{2m}^-(q)|}=\frac{|Z|}{|Y|}.
\]
Therefore, $Z=XY$.
\end{proof}

The factor pairs $(X,Y)$ in Row~2 of Table~\ref{TabOmega} will be constructed in Lemmas~\ref{LemOmega03}--\ref{LemOmega06} and~\ref{LemOmega11}.

\begin{lemma}\label{LemOmega03}
Let $Z=\Omega_7(q)$ with $m=3$, let $X=\N_1^\varepsilon[Z]^{(\infty)}=\Omega_6^\varepsilon(q)$ with $\varepsilon\in\{+,-\}$, and let $Y=\G_2(q)<Z$. Then $Z=XY$ with $X\cap Y=\SL_3^\varepsilon(q)$.
\end{lemma}

\begin{proof}
From~\cite[4.3.6]{Wilson2009} we see that $X\cap Y=\SL_3^\varepsilon(q)$. Hence
\[
\frac{|Z|}{|X|}=\frac{|\Omega_7(q)|}{|\Omega_6^\varepsilon(q)|}=q^3(q^3+\varepsilon 1)=\frac{|\G_2(q)|}{|\SL_3^\varepsilon(q)|}=\frac{|Y|}{|X\cap Y|},
\]
and so $Z=XY$.
\end{proof}

\begin{lemma}\label{LemOmega04}
Let $Z=\Omega_7(q)$ with $m=3$, let $M=\N_1^\varepsilon[Z]^{(\infty)}=\Omega_6^\varepsilon(q)$ with $\varepsilon\in\{+,-\}$, let $X=\Omega_5(q)<M$, and let $Y=\G_2(q)<Z$. Then $Z=XY$ with $X\cap Y=\SL_2(q)$.
\end{lemma}

\begin{proof}
Since $\Omega_6^\varepsilon(q)=\SL_4^\varepsilon(q)/\{\pm1\}$ and $\Omega_5(q)=\Sp_4(q)/\{\pm1\}$, by Lemma~\ref{LemOmega03},~\cite[Lemma~4.5]{LWX-Linear} and~\cite[Lemma~4.7]{LWX-Unitary} we have
\[
X\cap Y=X\cap(M\cap Y)=X\cap\SL_3^\varepsilon(q)=\SL_2(q).
\]
It follows that
\[
\frac{|Y|}{|X\cap Y|}=\frac{|\G_2(q)|}{|\SL_2(q)|}=q^5(q^6-1)=\frac{|\Omega_7(q)|}{|\Omega_5(q)|}=\frac{|Z|}{|X|},
\]
and thus $Z=XY$.
\end{proof}

\begin{lemma}\label{LemOmega05}
Let $Z=\Omega_7(q)$ with $m=3$, let $X=\Pa_1[Z]^{(\infty)}=q^5{:}\Omega_5(q)$, and let $Y=\G_2(q)<Z$. Then $Z=XY$ with $X\cap Y=[q^5]{:}\SL_2(q)$.
\end{lemma}

\begin{proof}
From~\cite[4.3.5]{Wilson2009} we see that $X\cap Y=[q^5]{:}\SL_2(q)$. Hence
\[
\frac{|Y|}{|X\cap Y|}=\frac{|\G_2(q)|}{|[q^5]{:}\SL_2(q)|}=q^6-1=\frac{|\Omega_7(q)|}{|q^5{:}\Omega_5(q)|}=\frac{|Z|}{|X|},
\]
and so $Z=XY$.
\end{proof}

\begin{lemma}\label{LemOmega06}
Let $Z=\Omega_7(q)$ with $m=3$, and let $Y=\G_2(q)<Z$. Then there exits a subgroup $X$ of $\Pa_1[Z]$ such that $X=q^4{:}\Omega^-_4(q)$ and $Z=XY$ with $X\cap Y=[q^3]$.
\end{lemma}

\begin{proof}
We extend $\beta$ to a plus type $8$-dimensional orthogonal space over $\bbF_q$ with a basis $x_1,\dots, x_8$ such that
\[
\beta(x_i,x_j)=\delta_{i+j,9} 
\]
for $i,j\in\{1,\dots,8\}$. Let $x=x_4+x_5$ and $y=x_4-x_5$. Then $x$ and $y$ are nonsingular, and
\[
x^\perp=\langle x_1,x_2,x_3,x_6,x_7,x_8,y\rangle.
\]
Thus we may take $Z=\Omega(x^\perp)$. Let $W_1=\langle x_2,x_3,x_6,x_7,y\rangle$ and $V_1=\langle x_1\rangle\oplus W_1$. 
Then $V_1=x_1^\perp\cap x^\perp$ and hence is stabilized by $Z_{\langle x_1\rangle}$. Since $x_1$ is singular, we have $Z_{\langle x_1\rangle}=\Pa_1[Z]$ and $Z_{x_1}=R{:}T$ with $R=q^5$ and $T=\Omega_5(q)$, where $R$ is the kernel of $Z_{x_1}$ acting on $V_1/\langle x_1\rangle$, and $T=Z_{x_1,x_8,W_1}=\Omega(W_1)$. 

Take $\mu\in\bbF_q^\boxtimes$, and let $w=x_3+\mu x_6$, $z=x_3-\mu x_6$, $E=R_w$, $S=T_w$ and
\[
X=R_w{:}T_w=E{:}S.
\]
Then $w^\perp\cap W_1=\langle x_2,x_7,y,z\rangle$ is a minus type $4$-dimensional orthogonal space, and so $S=\Omega_4^-(q)$.
For $a\in\bbF_q$, let $\sigma(a)$ be the linear transformation on $x^\perp$ satisfying
\[
\sigma(a)\colon x_3\mapsto x_3-ax_1,\ x_8\mapsto x_8+ax_6
\]
and fixing the vectors $x_1,x_2,x_6,x_7,y$. It is straightforward to verify that $\sigma(a)\in\mathrm{O}(x^\perp)$. 
Then since $|\sigma(a)|$ is odd, we deduce that $\sigma(a)\in\Omega(x^\perp)=Z$. 
Moreover, since $\sigma(a)$ fixes $x_1$ and acts trivially on $V_1/\langle x_1\rangle$, we obtain $\sigma(a)\in R$. Let $F$ be the group generated by $\sigma(a)$ with $a$ running over $\bbF_q$. 
Note that $R$ stabilizes $w+\langle x_1\rangle$ as $R$ is the kernel of $Z_{x_1}$ on $V_1/\langle x_1\rangle$.
Then since $F$ is regular on $w+\langle x_1\rangle$, we conclude that $R=R_w\times F$. This together with $R=q^5$ and $F=q$ implies that $E=R_w=q^4$.
Hence $X=E{:}S=q^4{:}\Omega_4^-(q)$.

For $a\in\bbF_q$, let $h(a)$ and $k(a)$ be the linear transformations on $x^\perp$ satisfying
\begin{align*}
h(a)&\colon x_7\mapsto x_7+ay+a^2x_2,\ y\mapsto y+2ax_2\\
k(a)&\colon x_3\mapsto x_3-ax_1,\ x_7\mapsto x_7-ay+a^2x_2,\ x_8\mapsto x_8+ax_6,\ y\mapsto y-2ax_2
\end{align*}
and fixing the remaining vectors in the basis $x_1,x_2,x_3,x_6,x_7,x_8,y$. It is straightforward to verify that $h(a)k(a)=\sigma(a)$ and $h(a)\in\mathrm{O}(x^\perp)_{w,x_1,x_8,W_1}$. Then since $|h(a)|$ is odd, we deduce that $h(a)\in\Omega(x^\perp)_{w,x_1,x_8,W_1}=T_w=S<X$.
Note that there are precisely two conjugacy classes of subgroups $\G_2(q)$ in $Z$, fused in $\mathrm{O}(x^\perp)$ (see~\cite[Table~8.40]{BHR2013}).
We may assume that $Y=\G_2(q)$ is the subgroup of $Z$ described in~\cite[4.3.4]{Wilson2009}, so that by~\cite[Equation~(4.34)]{Wilson2009} we have $k(a)\in Y$. Then as $k(a)$ fixes $x_1$, we obtain $k(a)\in Y_{x_1}$. Now it follows from $\sigma(a)=h(a)k(a)$ that $F\subset XY_{x_1}$, and so
\begin{equation}\label{EqnOmega03}
R=R_wF\subset X(XY_{x_1})=XY_{x_1}.
\end{equation}

Let $\overline{\phantom{x}}\colon Z_{x_1}\to Z_{x_1}/R$ be the quotient modulo $R$. Then $\overline{X}=\Omega_4^-(q)$. 
From~\cite[4.3.5]{Wilson2009} we see that $\overline{Y_{x_1}}=q^{1+2}{:}\SL_2(q)$. Consequently, $\overline{Z_{x_1}}=\overline{X}\,\overline{Y_{x_1}}$ by Lemma~\ref{LemOmega02} (it works the same for $m=2$). This together with~\eqref{EqnOmega03} yields $Z_{x_1}=XY_{x_1}$ by Lemma~\ref{LemXia01}. 
Since $Y$ is transitive on the set of singular vectors in $x^\perp$, it follows that $Z=Z_{x_1}Y=(XY_{x_1})Y=XY$, which then implies 
\[
|X\cap Y|=\frac{|X||Y|}{|Z|}=\frac{|q^4{:}\Omega_4^-(q)||\G_2(q)|}{|\Omega_7(q)|}=q^3,
\]
as required.
\end{proof}

In the next two lemmas we show the factor pairs $(X,Y)$ in Rows~3--4 of Table~\ref{TabOmega}.
For each $\varepsilon\in\{+,-\}$, there are precisely two conjugacy classes of subgroups of $\G_2(3^f)$ that are isomorphic to $\SL_3^\varepsilon(q)$ (see~\cite[Table~8.42]{BHR2013}).

\begin{lemma}\label{LemOmega07}
Let $Z=\Omega_7(3^f)$ with $p=3$, let $M=\G_2(q)<Z$, and let $Y=\N_1^\varepsilon[Z]^{(\infty)}=\Omega_6^\varepsilon(q)$ with $\varepsilon\in\{+,-\}$. Then there is precisely one conjugacy class of subgroups $X$ of $M$ isomorphic to $\SL_3^{-\varepsilon}(q)$ such that $Z=XY$.
For such a pair $(X,Y)$ we have $X\cap Y=q^2-1$.
\end{lemma}

\begin{proof}
By Lemma~\ref{LemOmega03} we have $M\cap Y=\SL_3^\varepsilon(q)$. Moreover, it is shown in the proof of~\cite[Proposition~A]{HLS1987} that there is precisely one conjugacy class of subgroups $X$ of $M$ isomorphic to $\SL_3^{-\varepsilon}(q)$ with $M=X(M\cap Y)$, and each such $X$ satisfies $X\cap(M\cap Y)=q^2-1$. Then since $X\cap(M\cap Y)=X\cap Y$ and
\[
\frac{|\SL_3^{-\varepsilon}(q)|}{q^2-1}=q^3(q^3+\varepsilon 1)=\frac{|\Omega_7(q)|}{|\Omega_6^\varepsilon(q)|}=\frac{|Z|}{|Y|},
\]
the lemma holds.
\end{proof}

\begin{lemma}\label{LemOmega08}
Let $Z=\Omega_7(3^f)$ with $p=3$ and $f$ odd, let $M=\G_2(q)<Z$, let $X={^2}\G_2(q)<M$, and let $Y=\N_1^+[Z]^{(\infty)}=\Omega_6^+(q)$. 
Then $Z=XY$ with $X\cap Y=\frac{q-1}{2}.2$.
\end{lemma}

\begin{proof}
By Lemma~\ref{LemOmega03} we have $M\cap Y=\SL_3(q)$. Moreover, it is shown in the proof of~\cite[Proposition~B]{HLS1987} that $X\cap(M\cap Y)=\frac{q-1}{2}.2$. Hence $X\cap Y=X\cap(M\cap Y)=\frac{q-1}{2}.2$, and so
\[
\frac{|X|}{|X\cap Y|}=\frac{|{^2}\G_2(q)|}{q-1}=q^3(q^3+1)=\frac{|\Omega_7(q)|}{|\Omega_6^+(q)|}=\frac{|Z|}{|Y|}.
\]
This implies that $Z=XY$.
\end{proof}

\begin{remark}
If we let $X={^2}\G_2(q)'$ in Lemma~\ref{LemOmega08}, then computation in \magma~\cite{BCP1997} shows that the conclusion $Z=XY$ would not hold for $q=3$. 
\end{remark}

The factor pairs $(X,Y)$ in Rows~5 and~6 of Table~\ref{TabOmega} are constructed in the following Lemmas~\ref{LemOmega09} and~\ref{LemOmega10}, respectively.

\begin{lemma}\label{LemOmega09}
Let $Z=\Omega_{13}(3^f)$ with $p=3$, let $X=\PSp_6(q)<Z$, and let $Y=\N_1^-[Z]^{(\infty)}=\Omega_{12}^-(q)$. Then $Z=XY$ with $X\cap Y=(\SL_2(q)\times\SL_2(q^2))/\{\pm1\}$.
\end{lemma}

\begin{proof}
We follow the setup in~\cite[4.6.3(a)]{LPS1990}. Let $\beta_0$ be a nondegenerate alternating form on the natural $6$-dimensional module $V_0$ preserved by $X$, let $E_1,F_1,E_2,F_2,E_3,F_3$ be a standard basis for $V_0$ with respect to $\beta_0$, and let
\[
V_1=\langle u\wedge w\mid u,w\in V_0,\,\beta_0(u,w)=0\rangle.
\]
Then $V_1$ is a $14$-dimensional submodule of the alternating square $\Lambda^2V_0$, and $V_1$ has a trivial $1$-dimensional $X$-submodule $V_2$ with $V_1/V_2$ irreducible. We may assume $V=V_1/V_2$ and
\[
\beta(u_1\wedge w_1,u_2\wedge w_2)=\beta_0(u_1,u_2)\beta_0(w_1,w_2)-\beta_0(u_1,w_2)\beta_0(u_2,w_1)
\]
for all $u_i,w_i$ among the basis vectors $E_1,F_1,E_2,F_2,E_3,F_3$. Let
\[
S=\langle E_1,F_1,E_2,F_2\rangle
\] 
and let $B=\N_1^-[Z]$ be a maximal subgroup of $Z$ containing $Y$. 
Then we may assume that $B=Z_{\langle u\rangle}$ with $u=E_1\wedge E_2+\mu F_1\wedge F_2$ for some $\mu\in\bbF_q$. 
Thus $Y=Z_u$, and it is shown in the proof of Lemma~A in~\cite[4.6.3]{LPS1990} that
\begin{equation}\label{EqnOmega01}
X\cap B=X_{\langle u\rangle}=(\SL_2(q)\times\SL_2(q^2).2)/\{\pm1\}
\end{equation}
with $(X\cap B)_S^S=\Sp(S)_{\langle u\rangle}=(2.\Omega_5(q))_{\langle u\rangle}=2.\mathrm{O}_4^-(q)=\SL_2(q^2).2$.
Now $X\cap Y=X_u$, and so $(X\cap Y)_S^S=\Sp(S)_u=(2.\Omega_5(q))_u=2.\Omega_4^-(q)=\SL_2(q^2)$ has index $2$ in $(X\cap M)_S^S$. 
This implies that $X\cap Y$ has index $2$ in $X\cap B$. Then we deduce from~\eqref{EqnOmega01} that 
\[
X\cap Y=(\SL_2(q)\times\SL_2(q^2))/\{\pm1\}.
\]
Hence
\[
\frac{|X|}{|X\cap Y|}=\frac{|\PSp_6(q)|}{|(\SL_2(q)\times\SL_2(q^2))/2|}=q^6(q^6-1)=\frac{|\Omega_{13}(q)|}{|\Omega_{12}^-(q)|}=\frac{|Z|}{|Y|},
\]
which leads to $Z=XY$.
\end{proof}

\begin{lemma}\label{LemOmega10}
Let $Z=\Omega(V)=\Omega_{25}(3^f)$ with $p=3$, let $X=\F_4(q)<Z$, and let $Y=Z_v$. Then $Z=XY$ with $Y=\Omega_{24}^-(q)$ and $X\cap Y=\mathrm{Spin}_8^-(q)=2.\Omega_8^-(q)$.
\end{lemma}

\begin{proof}
Clearly, $Y=\Omega_{24}^-(q)$. Since
\[
\frac{|X|}{|\mathrm{Spin}_8^-(q)|}=\frac{|\F_4(q)|}{|2.\Omega_8^-(q)|}=q^{12}(q^{12}-1)=\frac{|\Omega_{25}(q)|}{|\Omega_{24}^-(q)|}=\frac{|Z|}{|Y|},
\]
it suffices to prove $X\cap Y=\mathrm{Spin}_8^-(q)$, that is, $X_v=\mathrm{Spin}_8^-(q)$. By~\cite[Table~2]{CC1988} we have
\begin{equation}\label{EqnOmega02}
X_{\langle v\rangle}=\mathrm{Spin}^-_8(q).2=(2.\Omega_8^-(q)).2.
\end{equation}
Let $M$ be a maximal subgroup of $X$ containing $X_{\langle v\rangle}$. Then we deduce from $M\geqslant X_{\langle v\rangle}$ that $M_{\langle v\rangle}\geqslant X_{\langle v\rangle}$, while from $M\leqslant X$ we obtain $M_{\langle v\rangle}\leqslant X_{\langle v\rangle}$. Consequently, 
\[
M_{\langle v\rangle}=X_{\langle v\rangle}.
\]
Since $|M|\geqslant|X_{\langle v\rangle}|>q^{24}$, the Theorem of~\cite{LS1987} asserts that $M$ is either parabolic or one of the groups $\mathrm{Spin}_9(q)$, $\mathrm{Spin}^+_8(q).\Sy_3$, $^3{}\D_4(q).3$ and $\F_4(q^{1/2})$. Since $|M|$ is divisible by $q^8-1$, we conclude that 
\[
M=\mathrm{Spin}_9(q)=2.\Omega_9(q). 
\]
As $\bbF_q$ is a splitting field for $M$ (see~\cite[Page~241]{Steinberg1968}), we deduce from~\cite[Theorem~1.1]{Liebeck1983} that each irreducible submodule of $M$ acting on $V$ is either the natural module of dimension $9$ or the spin module of dimension $16$. Let $S$ be such a submodule. 
Then $S$ is a nondegenerate subspace of $V$, and so $V=S\oplus S^\perp$ with $S$ and $S^\perp$ both $M$-invariant. Write 
\[
v=x+y
\]
for some $x\in S$ and $y\in S^\perp$. Then $M_{\langle v\rangle}\leqslant M_{\langle x\rangle}\cap M_{\langle y\rangle}$ and $M_v=M_x\cap M_y$.

First assume that $S$ is the natural module of $M$. Then the action of $M$ on $S$ induces $\Omega(S)=\Omega_9(q)$ with kernel $2$.
Hence $(X_v)^S\leqslant(M_v)^S\leqslant(M_x)^S=\Omega(S)_x$, which together with~\eqref{EqnOmega02} implies that $X_v\neq X_{\langle v\rangle}$.
Since $Z_v$ has index $2$ in $Z_{\langle v\rangle}$, it follows that $X_v$ has index $2$ in $X_{\langle v\rangle}$. 
Thus we conclude from~\eqref{EqnOmega02} that $X_v=\mathrm{Spin}^-_8(q)$.

Next assume that $S$ is the spin module of $M$. Then $S^\perp$ has dimension $25-\dim(S)=9$.
Since $M=\mathrm{Spin}_9(q)=2.\Omega_9(q)$, the induced group of $M$ on $S^\perp$ is either trivial or equal to $\Omega_9(q)$.
For the latter case, replacing $S$ with $S^\perp$ in the previous paragraph gives $X_v=\mathrm{Spin}^-_8(q)$, as desired.
Suppose for the rest of the proof that the induced group of $M$ on $S^\perp$ is trivial. Then $M_{\langle y\rangle}=M$.
From~\eqref{EqnOmega02} we deduce that 
\[
M_{\langle x\rangle}\geqslant M_{\langle v\rangle}=X_{\langle v\rangle}=\mathrm{Spin}^-_8(q).2.
\]
If $x\neq0$, then the Proposition of~\cite[Appendix~3]{LPS1990} shows that $M_{\langle x\rangle}/2$ is one of the groups
\[
\Omega_7(q).2,\quad q^{6+4}{:}\SL_4(q).\tfrac{q-1}{2},\quad q^7.\G_2(q).\tfrac{q-1}{2},
\]
a contradiction.
Therefore, $x=0$. Hence $v=y$, and so 
\[
X_{\langle v\rangle}=M_{\langle v\rangle}=M_{\langle y\rangle}=M=\mathrm{Spin}_9(q),
\]
contradicting~\eqref{EqnOmega02}.
\end{proof}

\section{Sporadic cases of $(X,Y)$ in Table~\ref{TabOmega}}\label{SecOmega02}

In this section, we give the sporadic pairs $(X,Y)$ in Table~\ref{TabOmega}.
The subgroups $X$ in the following lemma is not $\Aut(Z)$-conjugate to those in Lemma~\ref{LemOmega04}.

\begin{lemma}\label{LemOmega11}
Let $Z=\Omega_7(3)$, let $M=\Sp_6(2)$ be a maximal subgroup of $Z$, and let $X=\Omega_5(3)<M$.
Then there is precisely one conjugacy class of maximal subgroups $Y$ of $Z$ isomorphic to $\G_2(3)$ such that $Z=XY$. 
For each such pair $(X,Y)$ we have $X\cap Y=\SL_2(3)$.
\end{lemma}

The factor pairs $(X,Y)$ in Rows~7--11 of Table~\ref{TabOmega} are constructed in Lemmas~\ref{LemOmega12}--\ref{LemOmega17} below.
They are verified by computation in \magma~\cite{BCP1997} except for the last two lemmas. 

\begin{lemma}\label{LemOmega12}
Let $Z=\Omega_7(3)$, and let $Y=\G_2(3)<Z$. Then $Z$ has precisely two (out of four) conjugacy classes of subgroups $X$ of the form $3^4{:}\Sy_5$ such that $Z=XY$. For each such pair $(X,Y)$ we have $X\cap Y=3^2$.
\end{lemma}

\begin{lemma}\label{LemOmega13}
Let $Z=\Omega_7(3)$, let $X$ be a subgroup of $Z$ such that $X=3^5{:}2^4{:}\A_5$ or $3^4{:}\A_6$ (there is a unique conjugacy class of such subgroups in $Z$ in either case), and let $Y=\G_2(3)<Z$. Then $Z=XY$ with 
\[
X\cap Y=
\begin{cases}
\mathrm{ASL}_2(3)&\textup{if }X=3^5{:}2^4{:}\A_5\\
3_+^{1+2}&\textup{if }X=3^4{:}\A_6.
\end{cases}
\]
\end{lemma}

\begin{lemma}\label{LemOmega14}
Let $Z=\Omega_7(3)$, let $X$ be a subgroup of $Z$ such that $X=3^3{:}\SL_3(3)$ (there is a unique conjugacy class of such subgroups in $Z$), and let $Y=\A_9$ or $\Sp_6(2)$. Then $Z=XY$ with 
\[
X\cap Y=
\begin{cases}
\Sy_3&\textup{if }X=\A_9\\
\GL_2(3)&\textup{if }X=\Sp_6(2).
\end{cases}
\]
\end{lemma}

\begin{lemma}\label{LemOmega15}
Let $Z=\Omega_7(3)$, let $X=\N_1^+[Z]^{(\infty)}=\Omega_6^+(3)$, and let $Y$ be a subgroup of $Z$ such that $Y=\A_9$ or $\Sp_6(2)$.
Then $Z=XY$ with
\[
X\cap Y=
\begin{cases}
2\times\Sy_5&\textup{if }Y=\A_9\\
2^4.\Sy_5&\textup{if }Y=\Sp_6(2).
\end{cases}
\]
\end{lemma}

\begin{lemma}\label{LemOmega16}
Let $Z=\Omega_7(3)$, let $Y$ be a subgroup of $Z$ such that $Y=\A_9$ or $\Sp_6(2)$.
Then there is precisely one conjugacy class of maximal subgroups $X$ of $Z$ isomorphic to $\G_2(3)$ such that $Z=XY$. 
For each such pair $(X,Y)$ we have
\[
X\cap Y=
\begin{cases}
\PSL_2(7)&\textup{if }Y=\A_9\\
2^3.\PSL_2(7)&\textup{if }Y=\Sp_6(2).
\end{cases}
\]
\end{lemma}

\begin{lemma}\label{LemOmega17}
Let $Z=\Omega_7(3)$, let $X$ be a subgroup of $Z$ as in the following table, where $c$ is the number of conjugacy classes of such subgroups in $Z$, and let $Y=\Pa_3[Z]=3^{3+3}{:}\SL_3(3)$. Then $Z=XY$ with $X\cap Y$ in the last row of the table.
\[
\begin{array}{|c|ccccc|}
\hline
X & 2^6{:}\A_7 & \Sy_8 & \A_9 & 2.\PSL_3(4) & \Sp_6(2) \\
\hline
c & 1 & 2 & 2 & 1 & 2 \\
\hline
X\cap Y & 6.\Sy_4 & \Sy_3\times\Sy_3 & 3^3{:}\Sy_3 & 3^2{:}4 & \SU_3(2){:}\Sy_3 \\
\hline
\end{array}
\]
\end{lemma}

\begin{lemma}\label{LemOmega18}
Let $Z=\Omega(V)=\Omega_9(3)$ with $(m,q)=(4,3)$, and let $Y=Z_v$. Then $T$ has precisely two (out of four) conjugacy classes of subgroups $S$ of the form $2.\Sy_5$ such that $T=ST_{U_1,e_1+U_1}$, while each subgroup $S$ of $T$ of the form $8.\A_5$ or $2^{1+4}.\A_5$ satisfies $T=ST_{U_1,e_1+U_1}$. For each such $S$ let $X=R{:}S$. Then $Z=XY$ with $X=3^{6+4}{:}S$, $Y=\Omega_8^-(3)$ and
\[
X\cap Y=
\begin{cases}
3^{3+3}{:}3&\textup{if }S=2.\Sy_5\\
3^{3+3}{:}\Sy_3&\textup{if }S=8.\A_5\\
3^{3+3}{:}\SL_2(3)&\textup{if }S=2^{1+4}.\A_5.
\end{cases}
\]
\end{lemma}

\begin{proof} 
It is clear that $X=3^{6+4}{:}S$ and $Y=\Omega_8^-(3)$. 
The statement on $T=ST_{U_1,e_1+U_1}$ follows from~\cite[Lemma~5.4]{LWX-Linear}, and for each such $S$ we obtain
\[
S_{U_1,e_1+U_1}=S\cap T_{U_1,e_1+U_1}=
\begin{cases}
3&\textup{if }S=2.\Sy_5\\
\Sy_3&\textup{if }S=8.\A_5\\
\SL_2(3)&\textup{if }S=2^{1+4}.\A_5.
\end{cases}
\]
From Lemma~\ref{LemOmega01} we obtain $X\cap Y=3^{3+3}.S_{U_1,e_1+U_1}$. Thus the conclusion on $X\cap Y$ holds, and so
\[
\frac{|X|}{|X\cap Y|}=\frac{|3^{6+4}{:}S|}{|3^{3+3}.S_{U_1,e_1+U_1}|}=3^4(3^4-1)=\frac{|\Omega_9(3)|}{|\Omega_8^-(3)|}=\frac{|Z|}{|Y|}.
\]
Hence $Z=XY$.
\end{proof}

\begin{lemma}\label{LemOmega19}
Let $Z=\Omega(V)=\Omega_{13}(3)$ with $(m,q)=(6,3)$, let $X=R{:}S$ with $S=\SL_2(13)<T$, and let $Y=Z_v$. 
Then $Z=XY$ with $X=3^{15+6}{:}\SL_2(13)$, $Y=\Omega_{12}^-(3)$ and $X\cap Y=3^{10+5}.3$.
\end{lemma}

\begin{proof} 
It is clear that $X=3^{15+6}{:}\SL_2(13)$ and $Y=\Omega_{12}^-(3)$.
From Lemma~\ref{LemOmega01} we obtain $X\cap Y=3^{10+5}.S_{U_1,e_1+U_1}$, and by~\cite[Lemma~5.5]{LWX-Linear} we have $S_{U_1,e_1+U_1}=3$.
Hence $X\cap Y=3^{10+5}.3$, and so
\[
\frac{|X|}{|X\cap Y|}=\frac{|3^{15+6}{:}\SL_2(13)|}{|3^{10+5}.3|}=3^6(3^6-1)=\frac{|\Omega_{13}(3)|}{|\Omega_{12}^-(3)|}=\frac{|Z|}{|Y|},
\]
which implies that $Z=XY$.
\end{proof}

\section{Proof of Theorem~\ref{ThmOmega}}

The following result will be needed.

\begin{lemma}\label{LemOmega20}
Let $A=\GO(V)_{\langle e_1\rangle}=\Pa_1[\GO(V)]$ with $m=3$ and $q\geqslant5$. Then $A$ has a unique conjugacy class of subgroups isomorphic to $\Omega_4^-(q)$.
\end{lemma}

\begin{proof}
Let $H$ be a subgroup of $A$ isomorphic to $\Omega_4^-(q)$, let $N$ be the kernel of $A$ acting on $\langle e_1\rangle^\perp/\langle e_1\rangle=\langle e_1,e_2,f_2,e_3,f_3,d\rangle/\langle e_1\rangle$, and let $\overline{\phantom{x}}\colon A\to A/N$ be the quotient modulo $N$. 
Then $N=q^5$, $\overline{A}=\GO_5(q)$, and $\overline{H}\cong H\cong\Omega_4^-(q)$.
Since $\overline{A}$ has a unique conjugacy class of subgroups isomorphic to $\Omega_4^-(q)$ (refer to~\cite[Table~8.22]{BHR2013}), we conclude that $\overline{H}$ stabilizes some nondegenerate minus type $4$-subspace, say $I$, of $e_1^\perp/\langle e_1\rangle$. Write
\[
u=e_3+\lambda f_3,\ \ w=e_3-\lambda f_3,\ \ W_1=\langle e_2,f_2,u,d\rangle. 
\]
Then $W_1$ is a nondegenerate minus type $4$-subspace of $V$, and so $W_1+\langle e_1\rangle$ is a nondegenerate minus type $4$-subspace of $e_1^\perp/\langle e_1\rangle$. By Witt’s lemma, there exists $x\in A$ such that 
\[
I^{\overline{x}}=W_1+\langle e_1\rangle.
\] 
Hence $\overline{H}^{\overline{x}}$ stabilizes $W_1+\langle e_1\rangle$, and so $H^x$ stabilizes $\langle e_1,e_2,f_2,u,d\rangle$. 

Let $M=(\Omega_2^+(q)\times\Omega_2^-(q)).[4]$ be a maximal subgroup of $H^x$. Note that $|M|$ is coprime to $p$, and $M\leqslant H^x$ stabilizes $\langle e_1,e_2,f_2,u,d\rangle$ and $\langle e_1\rangle$. By Maschke's theorem, there exists some $M$-invariant subspace $W_2$ of $\langle e_1,e_2,f_2,u,d\rangle$ such that 
\[
\langle e_1,e_2,f_2,u,d\rangle=\langle e_1\rangle\oplus W_2.
\]
If $W_2$ is not a nondegenerate minus type $4$-subspace of $V$, then $A_{W_2}$ has no subgroup isomorphic to $(\Omega_2^+(q)\times\Omega_2^-(q)).[4]$, contradicting $M\leqslant A_{W_2}$. Thus $W_2$ is a nondegenerate minus type $4$-subspace of $V$. Then by Witt’s lemma, there exists $y\in\GO(V)$ such that
\[
\langle e_1\rangle^y=\langle e_1\rangle\ \text{ and }\ W_2^y=W_1.
\]
It follows that $y\in A$, and $M^y$ stabilizes $W_1$. 
Since $\langle e_1,e_2,f_2,u,d\rangle$ is stabilized by both $H^x$ and $y$, it is stabilized by the subgroup $H^{xy}$ of $A$. Let 
\[
J=\langle e_1,e_2,f_2,u,d\rangle/\langle e_1\rangle. 
\]
Then $\overline{H^{xy}}\leqslant\overline{A}_J$, which together with $\overline{H^{xy}}\cong\Omega_4^-(q)\cong(\overline{A}_J)^{(\infty)}$ yields 
\[
\overline{H^{xy}}=(\overline{A}_J)^{(\infty)}. 
\]
Since $H^{xy}=(H^{xy})'\leqslant A'$, we see that $H^{xy}$ fixes $e_1$.

Let $\rho$ and $\sigma$ be the linear transformations on $V$ satisfying
\begin{align*}
\rho&\colon f_2\mapsto f_2-d-\tfrac{1}{2}e_2,\ d\mapsto d+e_2\\
\sigma&\colon e_2\mapsto-e_2,\ f_2\mapsto-f_2,\ e_3\mapsto-e_3,\ f_3\mapsto-f_3
\end{align*}
and fixing the other vectors in the basis $e_1,f_1,e_2,f_2,e_3,f_3,d$. 
It is straightforward to verify that $\rho$ and $\sigma$ are elements in $A$ stabilizing $J$. Since $|\rho|=p$ and $-1\in\Omega(\langle e_2,f_2,e_3,f_3\rangle)$, we thus obtain $\overline{\rho},\overline{\sigma}\in(\overline{A}_J)^{(\infty)}=\overline{H^{xy}}$. Accordingly, there exist $r$ and $s$ in $H^{xy}$ such that $\overline{r}=\overline{\rho}$ and $\overline{s}=\overline{\sigma}$. This means that $r$ and $s$ satisfy
\begin{align*}
r\colon&e_1\mapsto e_1,\ f_1\mapsto f_1+a_1e_1,\ e_2\mapsto e_2+a_2e_1,\ f_2\mapsto f_2-d-\tfrac{1}{2}e_2+a_3e_1,\\
&e_3\mapsto e_3+a_4e_1,\ f_3\mapsto f_3+a_5e_1,\ d\mapsto d+e_2+a_6e_1\\
s\colon&e_1\mapsto e_1,\ f_1\mapsto f_1+b_1e_1,\ e_2\mapsto-e_2+b_2e_1,\ f_2\mapsto-f_2+b_3e_1,\\
&e_3\mapsto-e_3+b_4e_1,\ f_3\mapsto-f_3+b_5e_1,\ d\mapsto d+b_6e_1
\end{align*}
for some $a_1,\dots,a_6,b_1,\dots,b_6\in\bbF_q$. Since $\overline{s}=\overline{\sigma}$ is an involution and $H^{xy}\cong\overline{H^{xy}}$, we derive that $s$ is an involution. In particular, $s^2$ fixes $f_1$, which implies $b_1=0$. Consequently, 
\begin{equation}\label{EqnOmega06}
f_1^{rs}=f_1+a_1e_1.
\end{equation}
Moreover, since $\overline{rs}$ maps $e_2+\langle e_1\rangle$ to $-e_2+\langle e_1\rangle$, we derive that $|rs|$ is even. This together with~\eqref{EqnOmega06} implies that either $a_1=0$, or $|rs|$ is divisible by $2p$. As $H^{xy}\cong\Omega_4^-(q)\cong\PSL_2(q^2)$ has no element of order divisible by $2p$, it follows that $a_1=0$. Then from 
\[
\beta(f_1^r,z^r)=\beta(f_1,z)=0\ \text{ for }\ z\in\{e_2,f_2,e_3,f_3,d\}
\]
we deduce that $a_2=a_3=a_4=a_5=a_6=0$. Hence $r=\rho$ stabilizes $W_1$. Since $|r|=|\rho|=p$ does not divide $|M^y|$, we have $r\notin M^y$, and so $\langle M^y,r\rangle=H^{xy}$ as $M^y$ is a maximal subgroup of $H^{xy}$. Then as $M^y$ and $r$ both stabilize $W_1$, we conclude that $H^{xy}$ stabilizes $W_1$, and so 
\[
H^{xy}<\mathrm{O}(\langle e_1,f_1,w\rangle)\times\mathrm{O}(W_1)=\mathrm{O}_3(q)\times\mathrm{O}_4^-(q).
\]
Since $H^{xy}\cong\Omega_4^-(q)$, it follows that $H^{xy}=\Omega(V)_{e_1,f_1,w}<A$. Thus $A$ has a unique conjugacy class of subgroups isomorphic to $\Omega_4^-(q)$.
\end{proof}

Let $G$ be an almost simple group with socle $L=\Omega_{2m+1}(q)$, and let $H$ and $K$ be nonsolvable subgroups of $G$ not containing $L$.
In Subsections~\ref{SecOmega01} and~\ref{SecOmega02} it is shown that all pairs $(X,Y)$ in Table~\ref{TabOmega} are factor pairs of $L$.
Hence by Lemma~\ref{LemXia02} we only need to prove that, if $G=HK$, then $(H,K)$ tightly contains $(X^\alpha,Y^\alpha)$ for some $(X,Y)$ in Table~\ref{TabOmega} and $\alpha\in\Aut(L)$. 
Suppose that $G=HK$. Then by~\cite[Theorem~5.1]{LWX} the triple $(L,H^{(\infty)},K^{(\infty)})$ lies in Table~\ref{TabInftyOmega}.

For $L=\Omega_7(3)$ computation in \magma~\cite{BCP1997} shows that $(H,K)$ tightly contains $(X^\alpha,Y^\alpha)$ for some $(X,Y)$ in Rows~1--4 or 7--9 of Table~\ref{TabOmega} and $\alpha\in\Aut(L)$. Thus we assume $L\neq\Omega_7(3)$ for the rest of the section.

\begin{table}[htbp]
\captionsetup{justification=centering}
\caption{$(L,H^{(\infty)},K^{(\infty)})$ for orthogonal groups in odd dimension}\label{TabInftyOmega}
\begin{tabular}{|l|l|l|l|l|l|}
\hline
Row & $L$ & $H^{(\infty)}$ & $K^{(\infty)}$ & Conditions\\
\hline
1 & $\Omega_{2m+1}(q)$ & $P.\SL_a(q^b)$ ($m=ab$), & $\Omega_{2m}^-(q)$ & $P\leqslant q^{m(m-1)/2}.q^m$\\
 & & $P.\Sp_a(q^b)$ ($m=ab$) & & \\
 \hline
2 & $\Omega_7(q)$ & $\Omega_6^+(q)$, $\Omega_6^-(q)$, $\Omega_5(q)$, $\Omega_4^-(q)$, & $\G_2(q)$ & \\
 & & $q^5{:}\Omega_5(q)$, $q^4{:}\Omega_4^-(q)$ & & \\ \hline
3 & $\Omega_7(3^f)$
& $\SU_3(3^f)$, ${^2}\G_2(3^f)'$ ($f$ odd) & $\Omega_6^+(3^f)$ & \\
\hline
4 & $\Omega_{13}(3^f)$ & $\PSp_6(3^f)$ & $\Omega_{12}^-(3^f)$ & \\
5 & $\Omega_{25}(3^f)$ & $\F_4(3^f)$ & $\Omega_{24}^-(3^f)$ & \\
\hline
6 & $\Omega_7(3)$ & $3^4{:}\A_5$, $3^5{:}2^4{:}\A_5$, $3^4{:}\A_6$ & $\G_2(3)$ & \\
7 & $\Omega_7(3)$ & $3^3{:}\SL_3(3)$, $\Omega_6^+(3)$, $\G_2(3)$ & $\A_9$, $\Sp_6(2)$ & \\
8 & $\Omega_7(3)$ & $2^6{:}\A_7$, $\A_8$, $\A_9$, $2.\PSL_3(4)$, $\Sp_6(2)$ & $3^{3+3}{:}\SL_3(3)$ & \\
9 & $\Omega_9(3)$ & $P.2^{1+4}.\A_5$, $P.\SL_2(5)$ & $\Omega_8^-(3)$ & $P\leqslant3^{6+4}$\\
10 & $\Omega_{13}(3)$ & $P.\SL_2(13)$ & $\Omega_{12}^-(3)$ & $P\leqslant3^{15+6}$\\
\hline
\end{tabular}
\vspace{3mm}
\end{table}

Recall that $R=q^{m(m-1)/2}.q^m$ is a special group with $\Z(R)=\Phi(R)=R'=q^{m(m-1)/2}$.

\begin{lemma}
Suppose that $H^{(\infty)}=P.S$ with $P\leqslant R$ and $S\leqslant T$, and $K^{(\infty)}=\Omega_{2m}^-(q)$, as in Row~$1$,~$9$ or~$10$ of Table~$\ref{TabInftyOmega}$. Then $(H,K)$ tightly contains some pair $(X,Y)$ in Row~$1$,~$4$,~$10$ or~$11$ of Table~$\ref{TabOmega}$.
\end{lemma}

\begin{proof}
Without loss of generality assume that $K^{(\infty)}=L_v$. Then $L_v\leqslant K\leqslant G_{\langle v\rangle}=\N_1^-[G]$.
Let $A$ be a maximal subgroup of $G$ containing $H$. By~\cite[Theorem~A]{LPS1990} and~\cite{LPS1996}, either $A=\Pa_m[G]$ or $m=3$ and $A\cap L=\G_2(q)$.
For the latter case, since $A=H(A\cap K)$, we derive from~\cite{HLS1987} that $q=3^f$ and $H^{(\infty)}=\SL_3(q)$, as in Row~4 of Table~\ref{TabOmega}.
For the rest of the proof we assume that $A=\Pa_m[G]$ is the subgroup of $G$ stabilizing $U$.

From Lemma~\ref{LemOmega01}(a) we see that the induced group of $H\cap K$ on $U$ stabilizes a hyperplane of $U$. Since the induced group of $H$ on $U$ is $H/P$, it follows that $H/P$ acts transitively on the set of hyperplanes in $U$. Hence~\cite[Theorem~1.2]{LWX-Linear} implies that $H/P$ has a normal subgroup in one of the following sets:
\[
\{\SL_a(q^b),\Sp_a(q^b)\}\ (m=ab),\ \{2.\Sy_5,8.\A_5,2^{1+4}.\A_5\}\ (L=\Omega_9(3)),\ \{\SL_2(13)\}\ (L=\Omega_{13}(3)).
\] 
In particular, $H/P$ acts irreducibly on $R/R'=q^m$.

Suppose that $P\leqslant R'$. Then $H$ is contained in a subgroup $M$ of $G$ such that 
\[
M\cap L=R'{:}T.\tfrac{q-1}{2}.
\]
It follows from Lemma~\ref{LemOmega01} that
\begin{equation}\label{EqnOmega04}
(R{:}T)\cap L_v=(R\cap L_v)'.q^{m-1}.(q^{m-1}{:}\SL_{m-1}(q))
\end{equation}
with $(R\cap L_v)'=q^{(m-1)(m-2)/2}$. Since $T\cap L_v\geqslant\SL(U)_{e_1,U_1}=\SL_{m-1}(q)$, we have
\begin{equation}\label{EqnOmega05}
(R'{:}T)\cap L_v\geqslant(R\cap L_v)'{:}(T\cap L_v)\geqslant(R\cap L_v)'{:}\SL_{m-1}(q).
\end{equation}
Since $G=HK=MK$, we have
\[
|M\cap K\cap L|_p\leqslant\frac{|M|_p|K|_p}{|G|_p}\leqslant\frac{|M|_p|K\cap L|_p}{|L|_p}=\frac{|M|_p}{q^m}=\frac{|M\cap L|_p|G/L|_p}{q^m}<\frac{|M\cap L|_p}{q^{m-1}}
\]
and
\[
|M\cap K\cap L|_p\geqslant\frac{|K|_p|M\cap L|_p}{|G|_p}\geqslant\frac{|K\cap L|_p|M\cap L|_p}{|L|_p|G/L|_p}=\frac{|M\cap L|_p}{q^m|G/L|_p}>\frac{|M\cap L|_p}{q^{m+1}}.
\]
As $|M\cap K\cap L|_p=|(R'{:}T)\cap L_v|_p$ and $|M\cap L|_p=q^{2m-2}|(R\cap L_v)'{:}\SL_{m-1}(q)|_p$, we obtain
\[
q^{m-3}|(R\cap L_v)'{:}\SL_{m-1}(q)|_p<|(R'{:}T)\cap L_v|_p<q^{m-1}|(R\cap L_v)'{:}\SL_{m-1}(q)|_p.
\]
However, it is clear from~\eqref{EqnOmega04} and~\eqref{EqnOmega05} that $(R{:}T)\cap L_v$ does not have such a subgroup $(R'{:}T)\cap L_v$, a contradiction.

Thus we conclude that $P\nleqslant R'$, and so $PR'/R'=R/R'$ as $H/P$ acts irreducibly on $R/R'$. Consequently, $R=PR'=P\Phi(R)$. 
Since $\Phi(R)$ consists of the non-generators of $R$, this implies that $P=R$. 
Hence $(H,K)$ tightly contains some pair $(X,Y)$ in Row~1,~10 or~11 of Table~\ref{TabOmega}, according to which normal subgroup $H/P$ has.
\end{proof}

\begin{lemma}
Suppose that $(L,H^{(\infty)},K^{(\infty)})$ lies in Row~$2$ of Table~$\ref{TabInftyOmega}$. Then $(H,K)$ tightly contains some pair $(X,Y)$ in Row~$2$ of Table~$\ref{TabOmega}$.
\end{lemma}

\begin{proof}
As $(L,H^{(\infty)},K^{(\infty)})$ lies in Row~2 of Table~\ref{TabInftyOmega}, we have $m=3$, $K^{(\infty)}=\G_2(q)$ and 
\[
H^{(\infty)}\in\{\Omega_6^+(q),\Omega_6^-(q),\Omega_5(q),\Omega_4^-(q),q^5{:}\Omega_5(q),q^4{:}\Omega_4^-(q)\}.
\]
Let $A$ be a maximal subgroup of $G$ containing $H$. By~\cite[Theorem~A]{LPS1990} and~\cite{LPS1996}, we have
\[
A\in\{\Pa_1[G],\N_1^+[G],\N_1^-[G],\N_2^+[G],\N_2^-[G]\}.
\]
If $H^{(\infty)}=\Omega_6^+(q)$, $\Omega_6^-(q)$ or $\Omega_5(q)$, then $(H,K)$ tightly contains the pair $(X,Y)$ in Lemma~\ref{LemOmega03} or~\ref{LemOmega04}. If $H^{(\infty)}=q^5{:}\Omega_5(q)$, then $(H,K)$ tightly contains the pair $(X,Y)$ in Lemma~\ref{LemOmega05}.

Next assume that $H^{(\infty)}=\Omega_4^-(q)$. If $A=\Pa_1[G]$, then by Lemma~\ref{LemOmega20} we may assume that $H^{(\infty)}$ pointwise fixes a nondegenerate $3$-subspace of $V$ and so $H\leqslant\N_3^-[G]$. However,~\cite[Theorem~A]{LPS1990} shows that $G\neq\N_3^-[G]K$. Thus $H\nleqslant\Pa_1[G]$. 
If $A=\N_1^+[G]$, then $A^{(\infty)}=\Omega_6^+(q)$ and $H^{(\infty)}<\N_2^-[A^{(\infty)}]<\N_3^-[G]$, which implies that $H\leqslant\N_3^-[G]$ and so $G=\N_3^-[G]K$, a contradiction. If $A=\N_1^-[G]$, then $A^{(\infty)}=\Omega_6^-(q)$, and either $H^{(\infty)}<\N_2^+[A^{(\infty)}]<\N_3^-[G]$ or $H^{(\infty)}<\Pa_1[A^{(\infty)}]<\Pa_1[G]$. However, neither of them is possible as they lead to $H\leqslant\N_3^-[G]$ and $H\leqslant\Pa_1[G]$ respectively.
If $A=\N_2^+[G]$ or $\N_2^-[G]$, then $A^{(\infty)}=\Omega_5(q)$ and $H^{(\infty)}<\N_1^-[A^{(\infty)}]<\N_3^-[G]$, again a contradiction. 

Now assume that $H^{(\infty)}=q^4{:}\Omega_4^-(q)$. Then $A=\Pa_1[G]$ or $\N_1^-[G]$. For the latter case we have $A^{(\infty)}=\Omega_6^-(q)$ and $H^{(\infty)}<\Pa_1[A^{(\infty)}]<\Pa_1[G]$, which implies that $H\leqslant\Pa_1[G]$. Hence it always holds $H\leqslant\Pa_1[G]$. Viewing Lemma~\ref{LemOmega20} we then conclude that $(H,K)$ tightly contains the pair $(X,Y)$ in Lemma~\ref{LemOmega06}.
\end{proof}

If $(L,H^{(\infty)},K^{(\infty)})$ does not lie in Rows~1--2 or~9--10 of Table~\ref{TabInftyOmega}, which are dealt with in the above two lemmas, then it lies in Rows~3--5 of the table as $L\neq\Omega_7(3)$. For these rows, the pair $(H,K)$ tightly contains $(X,Y)=(H^{(\infty)},K^{(\infty)})$ in Rows~3, 5 and~6 of Table~\ref{TabOmega}. This completes the proof of Theorem~\ref{ThmOmega}.

\section*{Acknowledgments}
The first author acknowledges the support of NNSFC grants no.~11771200 and no.~11931005. The second author acknowledges the support of NNSFC grant no.~12061083.

\end{document}